\documentclass[12pt]{amsart}
\usepackage{amsmath,amssymb, xypic,verbatim,amscd,color}
\usepackage{graphicx}
\usepackage{colordvi}

\numberwithin{equation}{section}

\newcommand\frg{\mathfrak{g}}

\newcommand\Res{\operatorname{Res}}

\newcommand\mon{\overline{\operatorname{M}}_{g,n}}

\newtheorem{theorem}{Theorem}[section]
\newtheorem{remark}[theorem]{ Remark}

\newtheorem{proposition}[theorem]{Proposition}

\newtheorem{definition}[theorem]{Definition}

\newtheorem{example}[theorem]{\bf Example}

\begin{document}
\title{Strange duality of Verlinde spaces for $G_2$ and $F_4$ }
\author{Swarnava Mukhopadhyay}
\address{Department of Mathematics\\ University of Maryland\\ CB \#4015, Mathematics Building
\\ College Park, MD 20742}
\email{swarnava@umd.edu}
\thanks{The first author was supported in part by NSF Grant \#DMS-0901249.}
\begin{abstract} We prove that the pull back of the canonical theta divisor for $E_8$-bundles at level one induces a strange duality between Verlinde spaces for $G_2$ and $F_4$ at level one on smooth curves of genus $g$. We also prove a parabolic generalization in terms of conformal blocks and write down identities between conformal blocks divisors in $\operatorname{Pic}(\mon)_{\mathbb{Q}}$. 
\end{abstract}
\maketitle 
\section{Introduction}
Let $G$ be a simple, simply-connected, complex algebraic group and $X$ be a smooth algebraic curve over $\mathbb{C}$ of genus $g$. Let $\mathcal{M}_{G}(X)$ denote the moduli stack of principal $G$ bundles on $X$. The works of \cite{KNR, LS} tell us that $\operatorname{Pic}(\mathcal{M}_G(X))=\mathbb{Z}\mathcal{L},$ where $\mathcal{L}$ is a ample generator of the Picard group. When $G=\operatorname{SL}_r$, the line bundle $\mathcal{L}$ is just the determinant of cohomology (see \cite{DN} for a definition). For other simple, simply-connected groups, the line bundle $\mathcal{L}$ is a root of the determinant of cohomology \cite{ KNR, LS, S}. The spaces of global sections of positive powers of this line bundle $H^0(\mathcal{M}_{G}(X),\mathcal{L}^{\otimes \ell})$
are known as Verlinde spaces and the positive integer $\ell$ is known as the level. The dimension of these spaces are given by  the famous Verlinde formula \cite{B, Fal, TUY}. By a result of Faltings \cite{Fal}, it follows that for $G=E_8$ and for any genus $g$, the dimension $\dim H^0(\mathcal{M}_{E_8}(X),\mathcal{L})=1$. This implies that the moduli stack $\mathcal{M}_{E_8}(X)$ carries a natural divisor $\Delta$. 

In this paper, we consider the subgroup $P=G_2\times  {F}_4$  of $E_8$. We study the pull back of the canonical $E_8$-divisor $\Delta$ under the map $\phi: \mathcal{M}_{G_2}(X)\times \mathcal{M}_{F_4}(X)\rightarrow \mathcal{M}_{E_8}(X)$. It is easy to check that $\phi^*{\mathcal{L}}\simeq \mathcal{L}_1\boxtimes \mathcal{L}_2$, where $\mathcal{L}_1$( respectively $\mathcal{L}_2$ ) is the ample generator of the Picard group of $\mathcal{M}_{G_2}(X)$ ( respectively $\mathcal{M}_{F_4}(X)$). The pull back $\phi^*\Delta$ gives a map well defined up to constants between the following spaces: 
\begin{equation}\label{SD}
\phi^*: H^0({\mathcal{M}_{G_2}(X),\mathcal{L}_1})^{\vee}\rightarrow H^0(\mathcal{M}_{F_4}(X),\mathcal{L}_2)
\end{equation}
Using the Verlinde formula, it is easy to see that the dimensions of both the spaces are same and is equal to the following:
$$\dim H^0(\mathcal{M}_{G_2},\mathcal{L}_1)=\dim H^0(\mathcal{M}_{F_4},\mathcal{L}_2)=\bigg( \frac{5+\sqrt{5}}{2}\bigg)^{g-1}+\bigg(\frac{5-\sqrt{5}}{2}\bigg)^{g-1}.$$ With this notation our first theorem is the following:
\begin{theorem}\label{main1}
The map $\phi^*$ in \eqref{SD} induces a strange duality isomorphism between the Verlinde spaces $H^0(\mathcal{M}_{G_2}(X),\mathcal{L}_1)$ and $H^0(\mathcal{M}_{F_4}(X),\mathcal{L}_2)$ for any smooth curve $X$ of genus $g$.
\end{theorem}
 “Strange duality” or “rank-level” duality for Verlinde spaces has been an interesting topic in the last few years. We refer the reader to the surveys \cite{ MOP1, Pau1, Po} for details. Strange duality for vector bundles was conjectured in the works A. Beauville, Donagi-Tu and was proved by P. Belkale \cite{Bel1, Bel2} and by Marian-Oprea \cite{MO}. For symplectic bundles, strange duality was proved by T. Abe \cite{A} (see also paper by P. Belkale \cite{Bel4}). For any maximal, semi-simple, simply connected subgroup $P$ of $E_8$ of maximal rank strange duality was proved by \cite{BP}. The subgroup $P=G_2\times F_4$ of $E_8$ that we consider in this paper is the only maximal subgroup of $E_8$ of non maximal rank such that the map of corresponding Lie algebras
 $\phi : \mathfrak{g}_2\oplus \mathfrak{f}_4\rightarrow \mathfrak{e}_8$ is a conformal embedding ( see Section \ref{conformal}). 
\begin{remark}
It is important to point out that the subgroup $P=G_2\times F_4$ of $E_8$ has no center. All the other conformal subgroups of $P$ that were studied in \cite{BP} has non trivial center. In \cite{BP}, the argument in the proof of strange duality makes critical use of the fact that the subgroup $P$ has a non trivial center. As pointed out in Section 7.2.2 of \cite{BP}, their argument breaks down when $P=G_2\times F_4$. 
\end{remark}

We prove a parabolic generalization of the above statement using the language of conformal blocks and rank-level duality. We now briefly recall the notion of conformal blocks and rank-level duality and state a more general theorem from which Theorem \ref{main1} follows. For details, we refer the reader to Section \ref{basic}. 

Let $\frg$ be a simple complex Lie algebra and $\ell$ be a positive integer. We denote by $P_{\ell}(\frg)$, the set of dominant integral weight $\lambda$ of $\frg$ such that $(\lambda ,\theta)\leq \ell$, where $\theta$ is the longest root of $\frg$ and $(,)$ is the normalized Cartan Killing form such that $(\theta, \theta)=2$. For each $\lambda \in P_{\ell}(\frg)$, there exists an highest weight integrable irreducible module $\mathcal{H}_{\lambda}(\frg,\ell)$ of the affine Lie algebra $\widehat{\frg}$ ( see Section \ref{basic} for a definition). 

For an $n$-tuple $\vec{\lambda}=(\lambda_1,\dots, \lambda_n)$ of level $\ell$ weights of $\frg$, we denote the tensor product $\mathcal{H}_{\lambda_1}(\frg,\ell)\otimes \dots \otimes \mathcal{H}_{\lambda_n}(\frg,\ell)$ by $\mathcal{H}_{\vec{\lambda}}(\frg,\ell)$. Now let $X$ be a stable curve of arithmetic genus $g$ and $p_1,\dots, p_n$ be $n$ distinct smooth marked points of $X$ with chosen formal coordinates $\mathfrak{X}=(\xi_1,\dots, \xi_n)$ around the marked points $(p_1,\dots, p_n)$. The algebra $\frg(\mathfrak{X}):=\frg\otimes H^0(X, \mathcal{O}_X(\vec{p}))$ acts on $\mathcal{H}_{\vec{\lambda}}(\frg,\ell)$ ( see Section \ref{basic}) by expansion of functions using formal coordinates. The space of Covacua or dual conformal blocks  $\mathcal{V}_{\vec{\lambda}}(\mathfrak{X},\frg,\ell)$ are defined as follows:
$$\mathcal{V}_{\vec{\lambda}}(\mathfrak{X},\frg,\ell)=\mathcal{H}_{\vec{\lambda}}(\frg,\ell)/ \frg(\mathfrak{X}) \mathcal{H}_{\vec{\lambda}}(\frg,\ell).$$ Its dual $\mathcal{V}_{\vec{\lambda}}^{\dagger}(\mathfrak{X},\frg,\ell)$ is known as conformal blocks. As $C$ varies over the moduli stack $\mon$, the moduli stack of genus $g$ curves with $n$ marked points, the space of dual conformal blocks form a vector bundle $\mathbb{V}_{\vec{\lambda}}(\frg,\ell)$ over $\mon$.  When the genus of the curve $g=0$, the top exterior powers of dual conformal block vector bundles give nef divisors known as conformal blocks divisors on $\overline{\operatorname{M}}_{0,n}$. Formulas for the first Chern class and the Chern character formulas of $\mathbb{V}_{\vec{\lambda}}(\frg,\ell)$ are known due to the works of \cite{Fakh, MOP, MOPPZ, Muk3}. When $n=0$ and $X$ is a smooth curve of genus $g$, the space of conformal blocks $\mathcal{V}^{\dagger}_{\vec{\lambda}}(\mathfrak{X},\frg,\ell)$ is canonically identified with the Verlinde space $H^0(\mathcal{M}_G(X),\mathcal{L}^{\otimes \ell})$. We refer the reader to the works of \cite{BL, Fal, KNR, LS}. A parabolic version of the above theorem is also known. 

Let $\mathfrak{p}=\mathfrak{p}_1\oplus \mathfrak{p}_2$ be a conformal subalgebra of $\frg$ (see Section \ref{conformal} for a definition) and denote by $(\ell_1,\ell_2)$  the Dynkin multi-index of the embedding $\phi: \mathfrak{p} \rightarrow \mathfrak{g}$. Let $\vec{\lambda}=(\lambda_1,\dots, \lambda_n)$ ( respectively $\vec{\mu}=(\mu_1,\dots,\mu_n)$) and $\vec{\Lambda}=(\Lambda_1,\dots,\Lambda_n)$ be $n$ tuples of level $\ell_1$ ( respectively $\ell_2$) and level one weights of $\mathfrak{p}_1$ (respectively $\mathfrak{p}_2$) and $\mathfrak{g}$ such that for each $1\leq i\leq n$, the $\widehat{\mathfrak{p}}_1\oplus \widehat{\mathfrak{p}}_2$-module $\mathcal{H}_{\lambda_i}(\frg_2,1)\otimes \mathcal{H}_{\mu_i}(\mathfrak{f}_4,1)$ appears with multiplicity one in the decomposition of the level one $\widehat{\frg}$-module $\mathcal{H}_{\Lambda_i}(\mathfrak{e}_8,1)$ as a $\widehat{\mathfrak{p}}_1\oplus \widehat{\mathfrak{p}}_2$-module. Functoriality of the embedding gives rise to the following map of conformal blocks:
\begin{equation}\label{ranklevel}
\mathcal{V}_{\vec{\lambda}}(\mathfrak{X}, \mathfrak{p}_1,\ell_1)\otimes \mathcal{V}_{\vec{\mu}}(\mathfrak{X}, \mathfrak{p}_2,\ell_2)\rightarrow \mathcal{V}_{\vec{\Lambda}}(\mathfrak{X}, \mathfrak{g},1).
\end{equation}
  We refer the reader to \cite{Muk1} for details. We now restrict to the conformal embedding $\frg_2\oplus \mathfrak{f}_4 \rightarrow \mathfrak{e}_8$. Let $\vec{\omega}_1$ (respectively $\vec{\omega}_4$) denote an $n$ tuple of $\omega_1$'s (respectively $\omega_4$'s).  With the above notation, we now state a generalization of Theorem \ref{main1}.
\begin{theorem}\label{main2}
Let $\vec{\Lambda}$ be an $n$-tuple of vacuum representations of $\mathfrak{e}_8$ at level one, then the map in \eqref{ranklevel} induces a rank level duality isomorphism between the spaces 
$$\mathcal{V}_{\vec{\omega}_1}(\mathfrak{X},\frg_2,1)\simeq \mathcal{V}^{\dagger}_{\vec{\omega}_4}(\mathfrak{X},\mathfrak{f}_4,1),$$
where $\mathfrak{X}$ as before denotes the choice of formal coordinates around the marked points on smooth curves $X$ of genus $g$. 
\end{theorem}
\begin{remark}
For $n=0$ and arbitrary conformal embeddings, the rank-level duality map \eqref{ranklevel} that arises in representation theory and the strange duality map commutes under the identification of conformal blocks with Verlinde spaces. This was first observed in \cite{Bel2} and was used in \cite{A, BP, Muk1}. The above immediately implies that for $n=0$, Theorem \ref{main2} is nothing but Theorem \ref{main1}. For arbitrary $n$, it is not clear how to give a geometric description of rank-level duality maps purely in the language of moduli stack of parabolic $G$ bundles on a curve $X$. 
\end{remark}
In Section \ref{proof}, we discuss the main reduction in the proof of the Theorem \ref{main2} with the key details in Section \ref{proofofmain}. We now discuss applications of the Theorem \ref{main2} in the birational geometry of $\mon$. Let $F(g,n)=(\frac{5+\sqrt{5}}{2})^{g-1}(\frac{1+\sqrt{5}}{2})^n+(\frac{5-\sqrt{5}}{2})^{g-1}(\frac{1-\sqrt{5}}{2})^n$. We have the following theorem:

\begin{theorem}\label{picard}Let $\vec{\omega}_1$ be an $n$-tuple of $\omega_1$'s and  $\vec{\omega}_4$ be an $n$-tuple of $\omega_4$'s. Then the following relation holds in $\operatorname{Pic}(\mon)_{\mathbb{Q}}$.

\begin{eqnarray*}
\bigg(4\lambda + \sum_{i=1}^n\psi_i\bigg)&=&\frac{1}{F(g,n)}(c_1(\mathbb{V}_{\vec{\omega}_1}(\frg_2,1)))+c_1(\mathbb{V}_{\vec{\omega}_4}(\mathfrak{f}_4,1))+\frac{F(g-1,n+2)}{F(g,n)}\delta_{irr}\\
&&+\sum_{h,A}\frac{F(h,|A|+1)F(g-h, n-|A|+1)}{F(g,n)}\delta_{h,A},
\end{eqnarray*}
where $\lambda$ is the first Chern class of the Hodge bundle, $\psi_i$ is the $i$-th Psi class, $\delta_{irr}$ denotes the class of the divisor corresponding to the irreducible nodal curves, $A$ is a subset of $\{1,\dots, n\}$ and $\delta_{h,A}$ denotes the boundary divisor corresponding to reducible nodal curves with one component having genus $h$ and containing the markings of the set $A$. In the above identity, the repetition $\delta_{h,A}=\delta_{g-h,A^c}$ is not allowed.
\end{theorem}

Finally in Section \ref{picardmon}, we give a proof of Theorem \ref{picard}. Using the coordinate free description of conformal blocks in \cite{Muk3}, we construct the rank-level duality map over $\mon$. Following the same methods in the geometric proof of the main theorem in \cite{Muk3}, we compute the order of vanishing of the rank-level duality map on the boundary of $\mon$. Once that is done, Theorem \ref{picard} follows easily.

We also point out that once the statement of Theorem \ref{picard} is formulated, it is possible to give a complete proof of Theorem \ref{picard} using the Chern class formulas of conformal block bundles \cite{Fakh, Muk3, MOP}. But we do not write this up in this paper. 

\subsection*{Acknowledgements} I thank Jeffrey Adams, Shrawan Kumar and Richard Wentworth for useful conversations during the preparation of this paper. This work was initiated by a question (see Theorem \ref{main1}) conveyed to the author by Christian Pauly. I thank him for his comments and suggestions. 

\section{Notations and basic definitions }\label{basic}In this section, we recall the basic definitions and notations that we will use in the rest of this paper. For the definitions and properties of conformal blocks, we refer the reader to \cite{TUY}. We refer the reader to \cite{Kac} for basic definitions in the theory of Kac-Moody Lie algebras. 

Let $\frg$ be a simple, complex Lie algebra and $\mathfrak{h}$ be a Cartan subalgebra. We can decompose the Lie algebra $\frg$ as follows:
$\frg=\mathfrak{h}\oplus \sum_{\alpha \in \Delta}\frg_{\alpha},$ where $\Delta=\Delta_+ \sqcup \Delta_{-}$ is a system of roots divided into positive and negative roots. We denote by $\theta$, the longest root in $\Delta$ and we also fix a Cartan Killing form $(,)$ on $\frg$ normalized such that $(\theta, \theta)=2$. 

\subsection{Affine Lie algebras and Representation Theory} Let $t$ be a formal variable, we define the affine Lie algebra $\widehat{\frg}$ to be the Lie algebra 
$\widehat{\frg}=\frg\otimes \mathbb{C}((t))\oplus \mathbb{C}c,$ where $c$ belongs to the center of the Lie algebra $\widehat{\frg}$ and the Lie bracket is given by the formula
$$[X\otimes f , Y\otimes g ] =[X,Y]\otimes fg + (X,Y) \Res_{t=0}gdf.c,$$ where $X$, $Y$, are elements of the Lie algebra $\frg$ and $f$ and $g$ are element of $\mathbb{C}((t))$. The finite dimensional Lie algebra $\frg$ embeds in the degree zero part as a Lie subalgebra of $\widehat{\frg}$. 

Let $\ell$ be a positive integer and let $P_{\ell}(\frg):=\{\lambda \in P_{+}(\frg) | (\lambda,\theta)\leq \ell\}$ denote the set of level $\ell$ dominant weights of $\frg$. For every $\lambda \in P_{\ell}(\frg)$ there exists an unique, irreducible, integrable, highest weight $\widehat{\frg}$-module $\mathcal{H}_{\lambda}(\frg,\ell)$. We now recall some important properties:
\begin{enumerate}
\item $\mathcal{H}_{\lambda}(\frg,\ell)$'s are infinite dimensional.
\item The finite dimensional $\frg$-module $V_{\lambda}$ is a subset of $\mathcal{H}_{\vec{\lambda}}(\frg,\ell)$. 
\item The spaces $\mathcal{H}_{\lambda}(\frg,\ell)$ are quotient of Verma module. If $v_{\lambda}$ is the unique highest weight of $V_{\lambda}$, then $v_{\lambda}$ is also the highest weight vector of $\mathcal{H}_{\lambda}(\frg,\ell)$.
\end{enumerate}
\subsection{Conformal Blocks} Let $X$ be a curve of arithmetic genus $g$ with $n$-marked points $\vec{p}=(P_1,\dots, P_n)$ with chosen formal coordinates $\mathfrak{X}=(\xi_1,\dots, \xi_n)$ satisfying the following properties:
\begin{enumerate}
\item The $X$ has at most nodal singularities.
\item The marked points $P_1,\dots, P_n$ are smooth. 
\item The curve $X-\{P_1,\dots, P_n\}$ is affine. (We do not need it strictly.)
\item A stability condition equivalent to the finiteness of the automorphism group. 
\end{enumerate}
For any positive integer $n$, we consider the new algebra $\widehat{\frg}_n:=\bigoplus_{i=1}^n\frg\otimes_{\mathbb{C}}\mathbb{C}((\xi_i))\oplus \mathbb{C}c$ with the obvious Lie bracket. The current algebra $\frg(\mathfrak{X})=\frg\otimes H^0(X, \mathcal{O}_X(\vec{p}))$ is a Lie subalgebra of $\widehat{\frg}_n$. This follows from the fact that the sum of the residues of a meromorphic function is zero. We are now ready to define conformal blocks. 

Let $\vec{\lambda}=(\lambda_1,\dots, \lambda_n)$ be an $n$-tuple of level $\ell$ weights of $\frg$. Let $\mathcal{H}_{\vec{\lambda}}:=\mathcal{H}_{\lambda_1}(\frg,\ell)\otimes  \dots \otimes \mathcal{H}_{\lambda_n}(\frg,\ell)$. We define the space of conformal blocks as follows:
$$\mathcal{V}^{\dagger}_{\vec{\lambda}}(\mathfrak{X},\frg,\ell):=\operatorname{Hom}(\mathcal{H}_{\vec{\lambda}}/\frg(\mathfrak{X})\mathcal{H}_{\vec{\lambda}},\mathbb{C})$$ The dual space $\mathcal{V}_{\vec{\lambda}}(\mathfrak{X},\frg,\ell)$ is known as the space of covacua. These are finite dimensional vector spaces and their dimensions are given by the Verlinde formula. 

\subsection{Properties of conformal blocks}We now recall various important properties of conformal blocks. 
\subsubsection{Gauge Symmetry} Let $Y$ be an element of the Lie algebra and $f \in H^0( X, \mathcal{O}_X(\vec{p}))$ and $\langle \Psi| \in \mathcal{V}^{\dagger}_{\vec{\lambda}}(\mathfrak{X},\frg,\ell)$ satisfies the following symmetry:
$\sum_{i=1}^n\langle \Psi| \rho_i(Y\otimes f(\xi_i) \Phi \rangle =0,$
where $|\Phi\rangle \in \mathcal{H}_{\vec{\lambda}}$ and $\rho_i$ is the action on the $i$-th component of $|\Phi\rangle$.  
\subsubsection{Propagation of Vacua} Let $P_{n+1}$ be the new marked point with chosen formal coordinate $\xi_{n+1}$ and $\mathfrak{X}'$ denotes the new data of formal coordinates. Then there is a natural isomorphism 
$$\mathcal{V}_{\vec{\lambda}}(\mathfrak{X},\frg,\ell) \simeq \mathcal{V}_{\vec{\lambda},\omega_0}(\mathfrak{X}',\frg,\ell),$$ where $\omega_0$ is the vacuum representation. This above isomorphism is given by a simple formula. We refer the reader to \cite{TUY} for exact details.

\subsubsection{Conformal Blocks in Family}Let $\mathcal{F}$ be a family of curves with $n$ marked points and chosen formal coordinates, where each curve in the family has at most nodal singularities and satisfies the condition described in the previous section. There exists locally free sheaves $\mathcal{V}^{\dagger}_{\vec{\lambda}}(\mathcal{F},\frg,\ell)$ commuting with base change. Moreover by \cite{Fakh, Tsu}, these spaces conformal blocks can be constructed without choice of formal coordinates and gives a vector bundle $\mathbb{V}_{\vec{\lambda}}(\frg,\ell)$ on $\mon$. 

\subsubsection{KZ/Hitchin connection } If $\mathcal{F}$ is a family of smooth projective curves, then the locally free sheaf $\mathcal{V}^{\dagger}_{\vec{\lambda}}(\mathcal{F},\frg,\ell)$ carries a flat projective connection. This connection is known as the KZ/Hitchin/WZW connection.

\section{Conformal Embeddings}\label{conformal}In this section, we recall basic details about conformal embeddings. We refer the reader to \cite{Kac} for further details. First we recall the notion of Dykin-index of an embedding. Let $\phi : \mathfrak{p} \rightarrow \mathfrak{g}$ be an embedding. We define the Dynkin index $d_{\phi}$ of the embedding $\phi$ to be the ratio of the normalized Cartan-Killing form, i.e.  
$$d_{\phi}\langle x, y\rangle_{\mathfrak{p}}=\langle \phi (x), \phi(y) \rangle_{\mathfrak{g}}$$ It is well known that $d_{\phi}$ is a positive integer. If $\frg$ is simple, we define the conformal anomaly and the trace anomaly to be  $$c(\frg,\ell)=\frac{\ell \dim \frg}{g^*+\ell} \ \mbox{and} \  \Delta_{\lambda}(\frg,\ell)=\frac{(\lambda,\lambda+2\rho)}{2(g^*+\ell)},$$ where $g^*$ is the dual Coxeter number of $\frg$, $\ell$ is a non negative integer and $\lambda$ is a level $\ell$ weight. 

 Let $\mathfrak{p}_1$, $\mathfrak{p}_2$ and $\frg$ be three simple Lie algebras and $\phi: \mathfrak{p}_1\oplus \mathfrak{p}_2 \rightarrow \frg$ be an  embedding of Lie algebras. Let $(\ell_1,\ell_2)$ denote the Dynkin multi-index of the embedding. We now define conformal embedding. 
\begin{definition}
A embedding $\phi$ is conformal if the difference of the conformal anomaly is zero, i.e. 
$c(\mathfrak{p}_1,\ell_1)+c(\mathfrak{p}_2,\ell_2)=c(\frg,1)$
\end{definition}
\begin{example}Many well known examples are conformal. We recall a few of them with their corresponding Dynkin multi-indices. 
\begin{itemize}
\item $\mathfrak{sl}(r) \oplus \mathfrak{sl}(s) \rightarrow \mathfrak{sl}(rs)$ with Dynkin multi-index $(s,r)$.
\item $\mathfrak{so}(p)\oplus \mathfrak{so}(q) \rightarrow \mathfrak{so}(pq)$ with Dynkin multi-index $(q,p)$.
\item $\mathfrak{sp}(2r)\oplus \mathfrak{sp}(2s) \rightarrow \mathfrak{so}(4rs)$ with Dynkin multi-index $(s,r)$. 
\item $\mathfrak{g}_2 \oplus \mathfrak{f}_4 \rightarrow \mathfrak{e}_8$ with Dynkin multi-index $(1,1)$. 
\end{itemize}
\end{example}
Conformal embedding have all been classified \cite{SW} and has been used to produce rank-level dualities of conformal blocks in all known cases. We refer the reader to the works of the  \cite{A}, \cite{BP}, \cite{Bel1}, \cite{Muk1} for further details. Now we recall following \cite{KW}, the main properties of conformal embeddings:
\begin{enumerate}
\item Let $\mathfrak{p}\rightarrow \frg$ be a conformal subalgebra. Then any level one highest weight, irreducible, integrable module $\mathcal{H}_{\Lambda}(\frg,1)$ decomposes into a finite direct sum of $\widehat{\mathfrak{p}}$-modules.
\item Let $\mathfrak{p}\rightarrow \frg$ be a conformal subalgebra. Then for any integer $k$, the actions of  $k$-th Virasoro operator for $\mathfrak{p}$ and $\frg$ acts are the same on $\operatorname{End}(\mathcal{H}_{\Lambda}(\frg,1)$. This property of conformal embeddings tells us that the rank-level duality maps defined using conformal embeddings are flat with respect to the KZ/Hitchin connection.  

\end{enumerate}

\subsection{Branching Rules for the embedding $\mathfrak{g}_2\oplus \mathfrak{f}_4 \rightarrow \mathfrak{e}_8$}In this section, we recall the branching rules of the conformal embedding $\mathfrak{g}_2 \oplus \mathfrak{f}_4 \rightarrow \mathfrak{e}_8$. The Dykin multi-index of the embedding is $(1,1)$. The level one weights of $\mathfrak{g}_2$ are given by $\{\omega_0,\omega_1\}$ and that of $\mathfrak{f}_4$ is given by $\{\omega_0,\omega_4\}$. The only level one weight of $\mathfrak{e}_8$ is $\omega_0$. We recall the following from \cite{KS}.

\begin{proposition} The level one basic representation of $\widehat{\mathfrak{e}}_8$ decomposes into the following as $\widehat{\mathfrak{g}}_2\oplus \widehat{\mathfrak{f}}_4$ modules. 
$$\mathcal{H}_{\omega_0}(\mathfrak{e}_8,1)\simeq \mathcal{H}_{\omega_0}(\mathfrak{g}_2,1)\otimes \mathcal{H}_{\omega_0}(\mathfrak{f}_4,1)\oplus \mathcal{H}_{\omega_1}(\mathfrak{g}_2,1)\otimes \mathcal{H}_{\omega_4}(\mathfrak{f}_4,1)$$
\end{proposition}

Let $|{\bf 0}\rangle$ be the highest weight vector for the trivial representation of $\mathfrak{e}_8$. It is easy to see that $|{\bf 0}\rangle$ is the highest weight vector of $\mathcal{H}_{\omega_0}(\mathfrak{e}_8,1)$ and also of the component $\mathcal{H}_{\omega_0}(\mathfrak{g}_2,1)\otimes \mathcal{H}_{\omega_0}(\mathfrak{f}_4,1)$
We now give explicit expressions for the highest vector of the component $\mathcal{H}_{\omega_1}(\mathfrak{g}_2,1)\otimes \mathcal{H}_{\omega_4}(\mathfrak{f}_4,1)$. We will use these explicit expression in our proof of the strange duality. 

\begin{proposition}There exists an unique root $\alpha$ of $\mathfrak{e}_8$ such that the highest weight vector of $\mathcal{H}_{\omega_1}(\mathfrak{g}_2,1)\otimes \mathcal{H}_{\omega_4}(\mathfrak{f}_4,1)$ considered as an element of $\mathcal{H}_{\omega_0}(\mathfrak{e}_8,1)$ is given by $X_{\alpha}(-1)|{\bf 0}\rangle$, where $X_{\alpha}$ is a non-zero element in the root space of $\alpha$. 

\end{proposition}

\begin{proof}
The Lie algebra $\mathfrak{e}_8$ can be considered as a $\mathfrak{g}_2\oplus \mathfrak{f}_4$-module under the adjoint action. By a result in \cite{JFA}, the decomposition is given as follows:
$$\mathfrak{e}_8\simeq \mathfrak{g}_2\otimes \mathbb{C}\oplus \mathbb{C}\otimes \mathfrak{f}_4 \oplus V_{\omega_1}(\mathfrak{g}_2)\otimes V_{\omega_4}(\mathfrak{f}_4),$$ where $V_{\lambda}$ denotes the finite dimensional irreducible representation of a Lie algebra with highest weight $\lambda$. We observe that in the above decomposition the weight space of weight $(\omega_1,\omega_4)$ is one dimensional. It follows that there exists a unique root $\alpha$ of $\mathfrak{e}_8$ such that $\alpha$ restricted to the Cartan subalgebra of $\mathfrak{g}_2\oplus \mathfrak{f}_4$ is $(\omega_1,\omega_4)$. 


This implies that the corresponding roots spaces $X_{\alpha}$, is in $V_{\omega_1}\otimes V_{\omega_4}$. Now working with the opposite Borel and repeating the same argument, we can see that $X_{-\alpha}$ is also an element of $V_{\omega_1}\otimes V_{\omega_4}$.

Since $X_{\alpha}$ is the highest weight vector of $V_{\omega_1}\otimes V_{\omega_4}$, we get $[x,X_{\alpha}]=0$ for all $x$ in the positive nilpotent of $\mathfrak{g}_2\oplus \mathfrak{f}_4$. For $x$ as above and $n\geq 0$, we consider the following:
\begin{eqnarray*}
x(n)X_{\alpha}(-1)|{\bf 0 \rangle}&=&X_{\alpha}(-1)x(n)|{\bf 0 \rangle} + [x(n),X_{\alpha}(-1)]|{\bf 0}\rangle\\
&=& [x,X_{\alpha}](n-1)|{\bf 0} \rangle
\end{eqnarray*}
But since $X_{\alpha}$ is the highest weight vector we get $[x,X_{\alpha}]$ is zero. Hence $x(n)X_{\alpha}(-1)|{\bf 0 \rangle}=0$. The argument for the $x$ in the opposite Borel of $\mathfrak{g}_2\oplus \mathfrak{f}_4$ is similar. 
\end{proof}

\begin{remark}
The highest root $\theta$ of $\mathfrak{e}_8$ restricts to the root $\omega_1$ of $\mathfrak{g}_2$ and $\omega_4$ on $\mathfrak{f}_4$. Further the root $2\alpha_1+3\alpha_2+4\alpha_3+6\alpha_4+5\alpha_5+4\alpha_6+3\alpha_7+\alpha_8$ of $E_8$ restricts to $\omega_1$ (adjoint representation) of $\mathfrak{f}_4$ and zero on $\mathfrak{g}_2$ and the root $\alpha_1+2\alpha_2+2\alpha_3+3\alpha_4+3\alpha_5+2\alpha_6+\alpha_7+\alpha_8$ restricts to zero on $\mathfrak{f}_4$ and $\omega_2$ (adjoint representation) of $\mathfrak{g}_2$. This can be checked by a computer using \cite{Lie}. I thank Jeffrey Adams for this.
\end{remark}

We denote a Cartan subalgebra of $\mathfrak{e}_8$ by $\mathfrak{h}$ and chose a Cartan subalgebra $\mathfrak{h}_1$ (respectively $\mathfrak{h}_2$) for $\mathfrak{g}_2$ (respectively $\mathfrak{f}_4$) such that $\phi: \mathfrak{h}_1\oplus \mathfrak{h}_2 \hookrightarrow \mathfrak{h}$.
\begin{proposition}\label{weird}
There exists a root $\beta$ of $\mathfrak{e}_8$ and an $H$ in $\mathfrak{h}\backslash \mathfrak{h}_1\oplus \mathfrak{h}_2$ such that $\beta(H) \neq 0$ and $X_{\beta}$ is not in the image of $\mathfrak{g}_2\oplus \mathfrak{f}_4$. 
\end{proposition}
\begin{proof}
First we claim that there exists a coroot $\beta^{\vee}$ of $\mathfrak{e}_{8}$ which is not in the image of $\mathfrak{h}_1\oplus \mathfrak{h}_2$. This is possible since the embedding is not of full rank. 
We can write $\beta^{\vee}=\phi(h_1)+\phi(h_2)+H$ such that $H$ is orthogonal to $\phi(\mathfrak{h}_1\oplus \mathfrak{h}_2)$. Now $\beta(H)=(H,H)\neq 0$. 

Next we claim that $X_{\beta}$ can not be in the image of $\frg_2 \oplus \mathfrak{f}_4$. If $X_{\beta}$ is in $\frg_2\oplus \mathfrak{f}_4$, by working with the opposite Borel, it follows that $X_{-\beta}$ is also in $\frg_2\oplus \mathfrak{f}_4$. Now $[X_{\beta},X_{-\beta}]$ is also an element of $\frg_2 \oplus \mathfrak{f}_4$. This is a contradiction, since $\beta^{\vee}$ is not in $\mathfrak{h}_1\oplus \mathfrak{h}_2$. This proves the proposition.

\end{proof}

\subsection{Description of the rank-level duality map }From the branching rules described in the previous section, we know that $\mathcal{H}_{\omega_0}\otimes \mathcal{H}_{\omega_0}$ and $\mathcal{H}_{\omega_1}\otimes \mathcal{H}_{\omega_4}$ appears in the decomposition of $\mathcal{H}_{\omega}(\mathfrak{e}_8,1)$. So we let $\vec{\lambda}=(\lambda_1,\dots,\lambda_n)$( respectively $\vec{\mu}$) be an $n$-tuple of level one weights of $\frg_2$ ( respectively $\mathfrak{f}_4$) and $\vec{\Lambda}$ be an $n$-tuple of $\omega_0$'s such that $\mathcal{H}_{\lambda_i}\otimes \mathcal{H}_{\mu_i}$ appears in the decomposition $\mathcal{H}_{\omega_0}(\mathfrak{e}_8,1)$. Taking tensor product, we get a map 
$$\otimes_{i=1}^n \mathcal{H}_{\lambda_i}(\frg_2,1)\otimes \mathcal{H}_{\mu_i}(\mathfrak{f}_4,1) \rightarrow \otimes_{i=1}^n \mathcal{H}_{\omega_0}(\mathfrak{e}_8,1)$$ 
Taking coinvariants we get a map between the following rank-level duality map between the spaces of covacua 
\begin{equation}\label{rld}
\mathcal{V}_{\vec{\lambda}}(\mathfrak{X},\frg_2,1)\otimes\mathcal{V}_{\vec{\mu}}(\mathfrak{X},\mathfrak{f}_4,1)\rightarrow \mathcal{V}_{\vec{\Lambda}}(\mathfrak{X},\mathfrak{e}_8,1),
\end{equation}
where $\mathfrak{X}$ is the data associated to a $n$-pointed nodal curve $X$ of arithmetic genus $g$ with formal coordinates around the marked points. It follows from \cite{Fal} that the dimension of $\mathcal{V}_{\vec{\Lambda}}(\mathfrak{X},\mathfrak{e}_8, 1)$ is one. Hence one can ask is the following map (well defined up to constants) an isomorphism
$$\mathcal{V}_{\vec{\lambda}}(\mathfrak{X},\frg_2,1)\rightarrow \mathcal{V}^{\dagger}_{\vec{\mu}}(\mathfrak{X},\mathfrak{f}_4,1).$$ 
In the next two sections, we discuss the proof Theorem \ref{main2}.

\section{The proof of Theorem \ref{main2}: Main reductions}\label{proof}
We now discuss the main reductions in the proof of Theorem \ref{main2}. We closely follow the general strategy of rank-level duality adopted in \cite{Muk1}. We will point out key differences in this special case:
\subsubsection{Equality of dimensions} It is shown in \cite{Fakh} that the dimension of the conformal block on $\mathbb{P}^1$ with $n$ marked point and weights $\omega_1$ (respectively $\omega_4$) for the Lie algebra $\frg_2$ (respectively $\mathfrak{f}_4$) is given by the $(n-1)$-th Fibonacci number $\operatorname{Fib}(n-1)$. Using this and the factorization ( see \cite{TUY}), we can show that the  the dimensions of the source and the target in Theorem \ref{main2} are the same. It is given by the following formula:
$$\dim\mathcal{V}_{\vec{\omega}_1}(\mathfrak{X},\frg_2,1)=\bigg( \frac{5+\sqrt{5}}{2}\bigg)^{g-1}\bigg(\frac{\sqrt{5}+1}{2}\bigg)^{n}+\bigg(\frac{5-\sqrt{5}}{2}\bigg)^{g-1}\bigg(\frac{1-\sqrt{5}}{2}\bigg)^n.$$ 
It is important to point out that we do not explicitly need the equality of the dimensions in the proof of Theorem \ref{main2}. 

\subsection{Reductions to genus zero}This step is similar to steps in \cite{BP, Muk1}. The key ingredients are the flatness of the rank-level duality under the KZ connection, factorization theorem  in \cite{TUY} and the compatibility of rank-level duality with factorization shown in \cite{BP}.  We divide this into various small steps. 

\subsubsection{} Equality of Virasoro operators tell us that the rank-level duality is flat under the $KZ/Hitchin$ connection. In particular, it implies that if rank-level duality holds for a special curve, it holds for all curves. 

\subsubsection{} First, we consider a one parameter family $\mathcal{F}$ of curves over $\operatorname{Spec}{\mathbb{C}}[[t]]$, where the generic fiber is a smooth curve $X$ of a fixed genus $g\geq 1$ and the special fiber is a nodal curve $X_0$ of arithmetic genus $g$. the normalization $\tilde{X_0}$ of $X_0$ is a smooth curve of genus $g-1$. The factorization theorem in \cite{TUY} identities conformal blocks on $X_0$ with a direct sum of conformal blocks over $\tilde{X_0}$. A sheaf theoretic version also holds. This is known as the sewing procedure. We refer the reader to \cite{BP, Muk1, TUY} for more details on sewing procedure.

\subsubsection{} In \cite{BP}, combining the sewing procedure in \cite{TUY}, degenerations of the rank-level duality maps along $\mathcal{F}$ have been studied. The level one weights of $\mathfrak{g}_2$ are $\omega_0, \omega_1$ and the level weights of $\mathfrak{f}_4$ are $\omega_0,\omega_4$. The only level one weight of $\mathfrak{e}_8$ is $\omega_0$. The branching rule for the conformal embedding $\mathfrak{g}_2\oplus \mathfrak{f}_4 \rightarrow \mathfrak{e}_8$ has the following important feature: Given a level one weight $\lambda$ of $\mathfrak{g}_2$, there exists a unique level weight $\mu$ of $\mathfrak{f}_4$ such that 
$\mathcal{H}_{\lambda}(\frg_2,1)\otimes \mathcal{H}_{\mu}(\mathfrak{f}_4,1)$ appears in the branching of $\mathcal{H}_{\omega_0}(\mathfrak{e}_8,1)$. This uniqueness property combined with compatibility results in \cite{BP} and in \cite{Muk1} guarantees that rank-level duality on smooth curves of genus $g$ reduces to smooth curves of genus $g-1$. Hence if we continue to repeat the process, we will reduce to show that rank-level duality holds on $\mathbb{P}^1$ with arbitrary number of marked points.

\subsection{Further reductions to the three pointed case}Let $X_0=X_1\cup X_2$ be a rational nodal curve with $n$ marked points and $X_1$ and $X_2$ are two components with $n_1$ and $n_2$ marked points isomorphic to $\mathbb{P}^1$. As before, we put $X_0$ in a family $\mathcal{F}$ such that the special fiber is $X_0$ and the generic fiber is $\mathbb{P}^1$ with $n$ marked points. The same methods as above further reduces the problem to proving rank-level duality on $\mathbb{P}^1$ with $\operatorname{max}(n_1,n_2)$ points. Repeating this process, we are reduced to show rank-level duality for $\mathbb{P}^1$ with three marked points.

\subsection{The three point case}The list of all possible rank-level dualities maps on $\mathbb{P}^1$ with three marked points is given in Section \ref{proofofmain}. All the conformal blocks that we handle on $\mathbb{P}^1$ with three points are one dimensional. This follows directly from the dimension formula. Thus to show that Theorem \ref{main2} holds it is enough to show all the rank-level duality maps on $\mathbb{P}^1$ with three marked points are non zero. The strategy to show this is similar to the strategy used in \cite{Muk1}. Let $\langle{\Psi}|$ be the unique non element of $\mathcal{V}^{\dagger}_{\omega_0, \omega_0, \omega_0}(\mathbb{P}^1, \mathfrak{e}_8,1)$. We explicitly construct vectors $v_1, v_2, v_3$ in $\mathcal{H}_{\omega_1}(\frg_2,1)\otimes \mathcal{H}_{\omega_4}(\mathfrak{f}_4,1)$ such that $\langle \Psi | v_1\otimes v_2\otimes v_3\rangle \neq 0$. We make use of the ``gauge-symmetry" of $\langle \Psi |$.


\section{Proof of the main theorem in the case of $\mathbb{P}^1$ with three marked points}\label{proofofmain}
In this section, we show that rank-level duality map between the following one dimensional conformal blocks is non-zero. 
\begin{enumerate}
\item $\mathcal{V}_{(\omega_0,\omega_0,\omega_0)}(\mathbb{P}^1,\mathfrak{g}_2,1)\otimes \mathcal{V}_{(\omega_0,\omega_0,\omega_0)}(\mathbb{P}^1,\mathfrak{f}_4,1) \rightarrow \mathcal{V}_{(\omega_0,\omega_0,\omega_0)}(\mathbb{P}^1,\mathfrak{e}_8,1)$
\item $\mathcal{V}_{(\omega_0,\omega_1,\omega_1)}(\mathbb{P}^1,\mathfrak{g}_2,1)\otimes \mathcal{V}_{(\omega_0,\omega_4,\omega_4)}(\mathbb{P}^1,\mathfrak{f}_4,1) \rightarrow \mathcal{V}_{(\omega_0,\omega_0,\omega_0)}(\mathbb{P}^1,\mathfrak{e}_8,1)$
\item $\mathcal{V}_{(\omega_1,\omega_1,\omega_1)}(\mathbb{P}^1,\mathfrak{g}_2,1)\otimes \mathcal{V}_{(\omega_4,\omega_4,\omega_4)}(\mathbb{P}^1,\mathfrak{f}_4,1) \rightarrow \mathcal{V}_{(\omega_0,\omega_0,\omega_0)}(\mathbb{P}^1,\mathfrak{e}_8,1)$
\end{enumerate}

\subsection{Strategy} We first discuss the strategy to handle the above cases. Let $\langle \Psi|$ denote the non zero element of $\mathcal{V}^{\dagger}_{(\omega_0,\omega_0,\omega_0)}(\mathbb{P}^1,\mathfrak{e}_8,1)$. Let $z$ be a global coordinate of $\mathbb{C}$ and without loss of generality assume that the three marked points are $1$, $\infty$ and $0$ with the obvious choice of local coordinates. 

If we can produce elements $|\Phi_1\rangle$, $|\Phi_2\rangle$ and $|\Phi_3\rangle$ such that $|\Phi_1\otimes \Phi_2 \otimes \Phi_3 \rangle \in \bigotimes_{i=1}^3\mathcal{H}_{\lambda_i}(\mathfrak{g}_2,1)\otimes \mathcal{H}_{\mu_i}(\mathfrak{f}_4,1) \subset \mathcal{H}_{\omega_0}(\mathfrak{e}_8,1)^{\otimes 3}$ and $\langle \Psi | \Phi_1\otimes \Phi_2 \otimes \Phi_3 \rangle \neq 0$, we will be done. Here $(\lambda_i,\mu_i)$ are level one weights that appear in the branching of the affine weight $\omega_0$ of $\mathfrak{e}_8$. 

We choose $|\Phi_2 \rangle$ to be the highest weight vector for the component $\mathcal{H}_{\lambda_2}(\mathfrak{g}_2,1)\otimes \mathcal{H}_{\mu_2}(\mathfrak{f}_4,1)$ and $|\Phi_3\rangle$ to be the highest weight of $\mathcal{H}_{\lambda_2}(\mathfrak{g}_2,1)\otimes \mathcal{H}_{\mu_2}(\mathfrak{f}_4,1)$ with respect to the opposite finite dimensional Borel. This is possible since the longest element of the Weyl group of $\mathfrak{e}_8$ restricts to the longest element of the Weyl group of the sub algebras. Final we choose $|\Phi_1\rangle$ such that $\mathfrak{h}$ weights of $|\Phi_1\otimes \Phi_2\otimes \Phi_3\rangle $ is zero, where $\mathfrak{h}$ is the Cartan subalgebra of $\mathfrak{e}_8$. We now use the Gauge condition for conformal blocks to show that $\langle \Psi | \Phi_1\otimes \Phi_2 \otimes \Phi_3 \rangle \neq 0$. We now complete the proof of the rank-level duality. The proof has been divided to three cases to deal with the three non equivalent cases discussed above.

\subsection{Case I} Let $\langle \Psi| $ be the nonzero element of $\mathcal{V}^{\dagger}_{(\omega_0,\omega_0,\omega_0)}(\mathbb{P}^1,\mathfrak{e}_8,1)$ and let as before $|\bf{0}\rangle$ denote a non zero element of the trivial representation of $\mathfrak{e}_8$. As discussed before $|\bf{0}\rangle$ is also the highest weight vector of $\mathcal{H}_{\omega_0}(\mathfrak{e}_8,1)$ and is also the highest vector of the component $\mathcal{H}_{\omega_0}(\mathfrak{g_2},1)\otimes \mathcal{H}_{\omega_0}(\mathfrak{f_4},1)$. As discussed in the strategy above we choose $|\Phi_1\rangle =|\Phi_2\rangle = |\Phi_3 \rangle = |\bf{0}\rangle $. Clearly by definition, we get $\langle \Psi | \bf{0}\otimes \bf{0}\otimes \bf{0}\rangle$ is non-zero.

\subsection{Case II} Recall that the highest weight vector of the component $\mathcal{H}_{\omega_1}(\mathfrak{g}_2,1)\otimes \mathcal{H}_{\omega_4}(\mathfrak{f}_4,1)$ is of the form $X_{\alpha}(-1)|\bf{0}\rangle$, where $\alpha$ is a root of $\mathfrak{e}_8$ and $X_{\alpha}$ is a non zero element in the $\alpha$-th root space of $\mathfrak{e}_8$. Since the opposite Borel of $\mathfrak{e}_8$ restrict to the opposite Borel of $\mathfrak{g}_2\oplus \mathfrak{f}_4$, we get that $X_{-\alpha}(-1)|\bf{0}\rangle$ is also in $\mathcal{H}_{\omega_1}(\mathfrak{g}_2,1)\otimes \mathcal{H}_{\omega_4}(\mathfrak{f}_4,1)$. 

Following our strategy, we choose $|\Phi_1 \rangle = | \bf{0}\rangle $, $|\Phi_2\rangle= X_{\alpha}(-1)|\bf{0}\rangle $ and $|\Phi_3\rangle = X_{-\alpha}(-1)|\bf{0}\rangle$. We now have the following: 
\begin{eqnarray*}
&&-\langle \Psi | {\bf 0}\otimes X_{\alpha}(-1) {\bf0} \otimes X_{-\alpha}(-1){\bf 0}\rangle\\
&&\ \ =\langle \Psi | X_{-\alpha} {\bf 0} \otimes X_{\alpha}(-1) {\bf 0} \otimes {\bf 0}\rangle + \langle \Psi | {\bf 0} \otimes X_{-\alpha}(1)X_{\alpha}(-1) {\bf0} \otimes {\bf 0}\rangle \ \ \mbox{( By Gauge condition)}\\
&& \ \ = 0 + \langle \Psi | {\bf 0} \otimes X_{-\alpha}(1)X_{\alpha}(-1) {\bf0} \otimes {\bf 0}\rangle\\
&& \ \ =\langle \Psi | {\bf 0} \otimes X_{\alpha}(-1) X_{-\alpha}(1) {\bf0} \otimes {\bf 0}\rangle + \langle \Psi | {\bf 0} \otimes [X_{-\alpha}(1), X_{\alpha}(-1)] {\bf0} \otimes {\bf 0}\rangle \\
&&\ \ =  \langle \Psi | {\bf 0} \otimes ([X_{-\alpha}, X_{\alpha}] + c\langle X_{-\alpha}, X_{\alpha}\rangle) {\bf0} \otimes {\bf 0}\rangle \\
&&\ \ = \langle X_{-{\alpha}}, X_{\alpha}\rangle \langle \Psi | {\bf 0 \otimes 0 \otimes 0 }\rangle\\
&& \ \ \neq 0
\end{eqnarray*}
\begin{remark}The same proof work if we choose any $\beta$ such that $X_{\beta}$ in $V_{\omega_1}\otimes V_{\omega_4}$

\end{remark}
\subsection{Case III}Let $\beta$ and $H$ be as in Proposition \ref{weird}. In this case $|\Phi_2\rangle =X_{\beta}(-1)|{\bf 0}\rangle$ and $|\Phi_3\rangle=X_{-\beta}(-1)|{\bf0}\rangle$. Now we choose $|\Phi_1(-1)\rangle =  H(-1)|{\bf 0} \rangle$. 
\begin{eqnarray*}
&&-\langle \Psi | H(-1){\bf 0}\otimes X_{\beta}(-1) {\bf0} \otimes X_{-\beta}(-1){\bf 0}\rangle\\
&&\ \ =\langle \Psi | X_{-\beta}H(-1) {\bf 0} \otimes X_{\beta}(-1) {\bf 0} \otimes {\bf 0}\rangle + \langle \Psi | H(-1) {\bf 0} \otimes X_{-\beta}(1)X_{\alpha}(-1) {\bf0} \otimes {\bf 0}\rangle\\
&& \ \ = \langle \Psi |(H(-1)X_{-\beta}+ [X_{-{\beta}}, H(-1)]){\bf 0}\otimes X_{\beta}(-1) {\bf 0} \otimes {\bf 0}\rangle \\
&&\ \ \hspace{2cm}+ \langle \Psi | H(-1) {\bf 0} \otimes X_{-\beta}(1)X_{\beta}(-1) {\bf0} \otimes {\bf 0}\rangle\\
&& \ \ =\langle \Psi |[X_{-{\beta}}, H(-1)]{\bf 0} \otimes X_{\beta}(-1) {\bf 0} \otimes {\bf 0}\rangle +\langle \Psi | H(-1) {\bf 0} \otimes X_{\beta}(-1) X_{-\beta}(1) {\bf0} \rangle\otimes {\bf 0}\rangle\\
&&\hspace{2cm} + \langle \Psi | H(-1) {\bf 0} \otimes [X_{-\beta}(1), X_{\beta}(-1)] {\bf0} \otimes {\bf 0}\rangle \\
&&\ \ = \langle \Psi |[X_{-{\beta}}, H](-1){\bf 0}\otimes X_{\beta}(-1) {\bf 0} \otimes {\bf 0}\rangle  + \langle \Psi | H(-1) {\bf 0} \otimes ([X_{-\beta}, X_{\beta}] + c\langle X_{-\beta}, X_{\beta}\rangle) {\bf0} \otimes {\bf 0}\rangle \\
&&\ \ = \langle \Psi |[X_{-{\beta}}, H](-1){\bf 0}\otimes X_{\beta}(-1) {\bf 0} \otimes {\bf 0} \rangle +\langle X_{-{\beta}}, X_{\beta}\rangle \langle \Psi | H(-1) {\bf 0 \otimes 0 \otimes 0 }\rangle\\
\end{eqnarray*} By the choice of $H$ and $X_{-\beta}$, it follows that $X_{-\beta}$ does not commute with $H$. 
It is easy to observe that the above expression up to a constant is equal to $\langle \Psi | X_{-\beta}(-1){\bf 0} \otimes X_{\beta}(-1){\bf 0} \otimes 0\rangle$. This by Case II is non zero. This completes the proof.
\section{Relations in the Picard group of $\mon$}\label{picardmon}Relations in $\operatorname{Pic}(\overline{\operatorname{M}}_{0,n})$ arising form rank-level duality has been studied in \cite{Muk3}. Using the geometric approach of \cite{Muk3}, we write down the relations in $\operatorname{Pic}(\mon)$ that arises from the strange duality considered in this paper. The main result of this section is the Proposition \ref{pic} from which Theorem \ref{picard} follows by taking the first Chern class. 

We now fix a few notations:
\begin{itemize}
\item Let $\mathcal{L}_i$ be the line bundle on $\mon$ whose fiber at a point $(X, P_1,\dots, P_n)$ is the cotangent space $T^*_{P_i}X$ to the curve $X$ at the point $P_i$. It's first Chern class is denoted by $\psi_i$. 
\item It is well known \cite{Fakh}, that the line bundle $\mathbb{V}_{\omega_0}(\mathfrak{e}_8,1)$ is the fourth tensor power of the Hodge bundle $\mathbb{H}$ on $\mon$ and we denote the Hodge class by $\lambda$. 
\item The divisor $\delta_{irr}$ denotes the class of the divisor corresponding to irreducible nodal curves. 
\item For any subset $A$ of $\{1,\dots, n\}$, the divisor $\delta_{h,A}$ corresponds to reducible nodal curves with one component having genus $h$ and containing the markings $A$. 
\end{itemize}
\begin{proposition}\label{pic} 
The line bundle $\det \mathbb{V}_{\vec{\omega}_1}(\mathfrak{g}_2,1)\otimes \bigotimes_{i=1}^n \mathbb{L}_i^{-F(g,n)}$ is isomorphic to the tensor product of following line bundles:
\begin{enumerate}
\item  $\det \mathbb{V}_{\vec{\omega}_4}(\mathfrak{f}_4,1)^{-1}\otimes \mathbb{H}^{\otimes 4F(g,n)}$, 
\item $\mathcal{O}_{\mon}(-F(g-1,n+2)\delta_{irr})$,
\item  $\mathcal{O}_{\mon}\big(-\sum_{h,A}F(h, |A|+1)F(g-h, n-|A|+1)\delta_{h,A}\big),$
\end{enumerate}
where $F(g,n)=(\frac{5+\sqrt{5}}{2})^{g-1}(\frac{1+\sqrt{5}}{2})^n+(\frac{5-\sqrt{5}}{2})^{g-1}(\frac{1-\sqrt{5}}{2})^n$, $\vec{\omega}_1$ is an $n$-tuple of $\omega_1$ and $\vec{\omega}_4$ is an $n$-tuple of $\omega_4$. Further $\delta_{h,A}=\delta_{g-h, A^c}$ to avoid repetition in the sum.

\end{proposition}
\begin{proof}By a straight forward calculation, we see that the trace anomalies $\Delta_{\omega_1}(\frg_2,1)=2/5$, $\Delta_{\omega_4}(\mathfrak{f}_4,1)=3/5$ and $\Delta_{\omega_0}(\mathfrak{e}_8,1)=0$. Hence the difference of trace anomalies 
$$n^{\omega_0}_{\omega_1,\omega_4}:=\Delta_{\omega_1}(\frg_2,1)+\Delta_{\omega_4}(\mathfrak{f}_4,1)-\Delta_{\omega_0}(\mathfrak{e}_8,1)=1$$
The coordinate free rank-level duality map in \cite{Muk3} and the above tells us that there is a map of the following vector bundles on $\mon$. 
$$\mathbb{V}_{\vec{\omega}_1}(\frg_2,1)\otimes \mathbb{V}_{\vec{\omega}_4}(\mathfrak{f}_4,1)\otimes \bigotimes_{i=1}^n \mathbb{L}_i^{-1} \rightarrow \mathbb{V}_{\vec{\omega}_0}(\mathfrak{e}_8,1).$$ We observe that the rank  of $\mathbb{V}_{\vec{\omega}_0}(\mathfrak{e}_8,1)$ is one. Taking determinants of the dual, we get a map of the following line bundles over $\mon$:
$$\det \mathbb{V}_{\vec{\omega}_1}(\frg_2,1) \otimes \bigotimes_{i=1}^n \mathbb{L}_i^{-\otimes F(g,n)}\rightarrow \det \mathbb{V}_{\vec{\omega}_4}(\mathfrak{f}_4,1)^{-1}\otimes\mathbb{V}_{\vec{\omega}_0}(\mathfrak{e}_8,1)^{\otimes F(g,n)}$$
Now rank-level duality implies that the above map is an isomorphism for every smooth curve of genus $g$. Now if we can compute the order of vanishing of the above map along $\delta_{irr}$ and $\delta_{h,A}$ we will done.

\subsubsection{Vanishing along $\delta_{irr}$} Now we observe that if $X_0$ is a irreducible nodal curve in $\delta_{irr}$ and $\tilde{X}_0$ is a normalization of $X_0$, then the factorization theorem in \cite{TUY} tells us 
$$\mathcal{V}_{\vec{\omega}_1}(X_0,\frg_2,1)\simeq \mathcal{V}_{\vec{\omega}_1,\omega_0,\omega_0}(\tilde{X}_0,\frg_2,1)\oplus \mathcal{V}_{\vec{\omega}_1,\omega_1,\omega_1}(\tilde{X}_0,\frg_2,1).$$ Similarly for $\mathfrak{f}_4$, we get 
$$\mathcal{V}_{\vec{\omega}_4}(X_0,\mathfrak{f}_4,1)\simeq \mathcal{V}_{\vec{\omega}_4,\omega_0,\omega_0}(\tilde{X}_0,\mathfrak{f}_4,1)\oplus \mathcal{V}_{\vec{\omega}_4,\omega_4,\omega_4}(\tilde{X}_0,\mathfrak{f}_4,1).$$ In both cases, the dimension equality works as $F(g,n)=F(g-1,n+2)+F(g-1,n)$. By the proof of main theorem \cite{Muk3} and using the above dimensions calculations, it is easy to see that the order of vanishing along $\delta_{irr}$ is $F(g-1,n+2)$. 

The proof of vanishing along $\delta_{h,A}$ is similar. We omit the details. 
\end{proof}
\newpage
\bibliographystyle{plain}
\def\noopsort#1{}

\end{document}